\documentclass[a4paper,11pt]{amsart}
\usepackage[english]{babel}
\usepackage{amssymb}
\usepackage{amscd}
\usepackage{mathtools}
\usepackage[paper=a4paper,left=3cm,right=3cm,top=2cm,bottom=2.0cm]{geometry}
\usepackage{enumerate}
\usepackage{hyperref}
\usepackage{tikz}
\usepackage{tikz-cd}
\usepackage{tasks}
\usepackage{tcolorbox}

\newtheorem{thm}{Theorem}[section]
\newtheorem*{thm-non}{Theorem}

\newtheorem{lem}[thm]{Lemma}
\newtheorem{prop}[thm]{Proposition}
\newtheorem{cor}[thm]{Corollary}
\theoremstyle{definition}
\newtheorem{defi}[thm]{Definition}
\newtheorem{rem}[thm]{Remark}
\newtheorem{exam}[thm]{Example}

\DeclareMathOperator\PP{\mathbb{P}}
\DeclareMathOperator\OO{\mathcal{O}}
\DeclareMathOperator\De{D^b}

\DeclareMathOperator\Hom{Hom}
\DeclareMathOperator\Ext{Ext}

\DeclareMathOperator\M{M}
\DeclareMathOperator\NS{NS}

\DeclareMathOperator\IZ{\mathcal{I}_{\mathcal{Z}}}

\DeclareMathOperator\cok{coker}

\DeclareMathOperator\rk{rk}

\newcommand{\ZZ}{\ensuremath{\mathbb{Z}}} 
\newcommand{\CC}{\ensuremath{\mathbb{C}}} 
\newcommand{\X}{\ensuremath{X^{[k]}}}
\newcommand{\Xk}{\ensuremath{X^{[k]}}}

\begin{document}
	
	\title{Stability and certain $\mathbb{P}^n$-functors}

	\author{Fabian Reede}
	\address{Institut f\"ur Algebraische Geometrie, Leibniz Universit\"at Hannover, Welfengarten 1, 30167 Hannover, Germany}
	\email{reede@math.uni-hannover.de}
	
	\keywords{stable sheaves, integral functors, Hilbert schemes}
	
	\subjclass[2010]{Primary: 14F08; Secondary: 14F08, 14D20, 14J60, 53C26}
	
	\begin{abstract}
Let $X$ be a K3 surface. We prove that Addington's $\mathbb{P}^n$-functor between the derived categories of $X$ and the Hilbert scheme of points $\Xk$ maps stable vector bundles on $X$ to stable vector bundles on $\Xk$, given some numerical conditions are satisfied.
	\end{abstract}
	
	\maketitle
	
	\section*{Introduction}
Moduli spaces of stable sheaves on K3 surfaces are examples of hyperk\"ahler manifolds, also known as compact irreducible holomorphic symplectic manifolds. These arise in the classification of compact K\"ahler  manifolds with trivial first Chern class and are therefore interesting objects to study. It is natural to wonder if moduli spaces of stable sheaves on hyperk\"ahler manifolds of higher dimension also have interesting properties.

Unfortunately moduli spaces of stable sheaves on higher dimensional varieties behave badly in general. Furthermore there are not many explicit examples of stable sheaves on higher dimensional hyperk\"ahler manifolds. There are the tautological bundles on Hilbert schemes of points on a K3 surface and on generalized Kummer varities, see \cite{rz3}, \cite{sch10}, \cite{stapleton} and \cite{wandel}. A second class is given by the ``wrong-way" fibers of a universal family of stable vector bundles on a moduli space of stable vector bundles, see \cite{rz}, \cite{rz2} and \cite{wray}.

In this article we want to give a new class of stable sheaves on Hilbert schemes of points. For this we use a result by Addington in \cite{adding}, saying that the integral functor
\begin{equation*}
\Phi: \De(X) \rightarrow \De(\Xk)
\end{equation*}
with kernel the universal ideal sheaf $\IZ$ on $X\times \Xk$ is a $\PP^{k-1}$-functor. 

Our main results can be summarized as follows:

\begin{thm-non}
Let $X$ be a K3 surface with $\NS(X)=\mathbb{Z}h$ and assume $E$ is a $\mu_h$-stable locally free sheaf with Mukai vector $v=(r,h,s)$  such that 
\begin{equation*}
\chi(E)\geqslant \frac{v(E)^2}{2}+(r+1)k+1.
\end{equation*}
Then the image of $E$ under the integral functor $\Phi: \De(X)\rightarrow \De(\Xk)$ is a $\mu_H$-stable locally free sheaf on $\Xk$ (for some ample class $H\in \NS(\Xk)$). 

Choosing $v=(r,h,s)$ such that $v^2+2<2r$ implies that all sheaves classified by the moduli space $\M_{X,h}(v)$ are locally free. If furthermore $\M_{X,h}(v)$ is a fine moduli space, then the integral functor $\Phi$ restricts to a morphism
\begin{equation*}
f: \M_{X,h}(v) \rightarrow \mathcal{M},\,\,\,[E]\mapsto [\Phi(E)] 
\end{equation*}
which identifies $\M_{X,h}(v)$ with a smooth connected component of a certain moduli space $\mathcal{M}$ of stable sheaves on $\Xk$.
\end{thm-non}	

The main theorem is not void, that is there are Mukai vectors on certain K3 surfaces that satisfy all conditions stated in the theorem. We will give two examples in \ref{exmp}.

This result is in the same vein as Yoshioka's results in \cite{yocomp}. There he starts with an arbitrary K3 surface together with an isotropic Mukai vector $v$ such that $\M_{X,h}(v)$ is also a fine moduli space with universal family $\mathcal{E}$. It is well known that $Y=\M_{X,h}(v)$ is again a K3 surface. He then proves that the Fourier-Mukai transform $\Phi_{\mathcal{E}}: \De(X)\rightarrow \De(Y)$, an equivalence in this case, preserves stability (even S-equivalence classes), given that some numerical conditions are satisfied.

All objects in this text are defined over the field of complex numbers $\mathbb{C}$. 

\subsection*{Acknowledgement}
I thank Ziyu Zhang for many useful discussions. I am also grateful to the anonymous
referees who helped to improve the presentation of the manuscript greatly.
	
	\section{\texorpdfstring{Background on $\mathbb{P}^{n}$-functors }%
		{Background on Pn-functors}}
We start by recalling some basic facts about the $\mathbb{P}^n$-functors we are interested in. For the general definition of $\mathbb{P}^n$-functors we refer to Addington's paper, see \cite[Section 4]{adding}. For similar results see also \cite[Theorem 1.1]{mark2} and \cite{log}.

\begin{defi}
Let $X$ be a K3 surface. Define the integral functor $\Phi$ by
		\begin{equation*}\label{eqn:defPhi}
			\mathrm{\Phi} \colon \De(X) \longrightarrow \De(\X),\,\,\,E\mapsto Rp_{*}(q^{*}E\otimes\IZ).
		\end{equation*}
		Thus $\Phi$ has the universal ideal sheaf $\IZ$ on $X \times X^{[k]}$ as kernel. Here $p: X\times \Xk \rightarrow \Xk$ and $q: X\times \Xk \rightarrow X$ are the projections. We note that $\Phi$ is a $\PP^{k-1}$-functor with associated autoequivalence $H=[-2]$, see \cite[Theorem 3.1, Example 4.2(2)]{adding}. 
\end{defi}

	\begin{rem}\label{pn-func2}
		The fact that the integral functor $\mathrm{\Phi}$ is a $\mathbb{P}^{k-1}$-functor with associated autoequivalence $H=[-2]$ has the following helpful consequence: for any two elements $E, F\in \De(X)$ there is an isomorphism of graded vector spaces
		\begin{equation*}\label{eqn:pnfunc}
			\Ext^{*}_{\X}(\mathrm{\Phi}(E),\mathrm{\Phi}(F))\cong \Ext^{*}_X(E,F)\otimes \mathrm{H}^{*}(\PP^{k-1},\CC),
		\end{equation*}
	see for example \cite[Section 2.1]{add16-2}.
	\end{rem}
	
	\begin{rem}\label{extiso}
	Remark \ref{pn-func2} shows that if $\Ext^i_X(E,F)=0$ for $i<0$ then the natural maps
		\begin{equation*}
			\Ext^{i}_X(E,F) \rightarrow \Ext^{i}_{\X}(\mathrm{\Phi}(E),\mathrm{\Phi}(F))
		\end{equation*}
		are isomorphisms for $i=0$ and $1$. This especially applies to the case that $E$ and $F$ are in fact sheaves (interpreted as complexes concentrated in degree zero in $\De(X)$).
	\end{rem}

	\section{Preservation of slope-stability for vector bundles}\label{sect1}
	
	Throughout this article we assume that $X$ is a K3 surface with $\NS(X) = \ZZ h$, where $h$ is a primitive ample class.

	\begin{lem}\label{cohvan}
		Let $E$ be a $\mu_h$-stable vector bundle on $X$ with Mukai vector $v=(r,h,s)$ such that
	\begin{equation}\label{ineq}
		\chi(E) \geqslant \frac{v(E)^2}{2}+(r+1)k+1,
	\end{equation}
then $\mathrm{H}^i(X,E\otimes I_Z)=0$ for $i=1,2$ and all $[Z]\in\Xk$.
	\end{lem}
\begin{proof}
Note that $I_Z$ is $\mu_h$-stable with $\mu_h(I_Z)=0$ so
\begin{equation*}
\mathrm{H}^2(E\otimes I_Z)\cong \Ext^2_X(E^{*},I_Z)\cong \Hom_X(I_Z,E^{*})^{\vee}=0
\end{equation*}
by \cite[Proposition 1.2.7]{huy} as $\mu_h(E^{*})<0$.

Next we prove that $\mathrm{H}^1(E\otimes I_Z)=0$ for all $[Z]\in \Xk$, following the proof of \cite[Proposition 2.2.]{oprea}. 

Assume $\mathrm{H}^1(X,E\otimes I_Z)\neq 0$. By Serre duality we find
\begin{equation*}
	\mathrm{H}^1(X,E\otimes I_Z)\cong \Ext^1_X(E^{*},I_Z)\cong \Ext^1(I_Z,E^{*})^{\vee}.
\end{equation*}
Thus there is a nontrivial extension
		\begin{equation*}
	\begin{tikzcd}
		0 \arrow[r] & E^{*} \arrow[r] &  G \arrow[r] & I_Z \arrow[r] & 0.
	\end{tikzcd}
\end{equation*}
Using \cite[Lemma 2.1]{yosh} shows that $G$ is also $\mu_h$-stable with Mukai vector 
\begin{equation*}
	v(G)=v(E^{*})+v(I_Z).
\end{equation*}
The Euler characteristic of a pair $(F_1,F_2)$ of coherent sheaves can be expressed as
\begin{equation*}
	\chi(F_1,F_2)=\sum\limits_{i=0}^2 (-1)^i \mathrm{ext}^i_X(F_1,F_2)=-\left\langle v(F_1),v(F_2) \right\rangle.
\end{equation*}
We can now compute
\begin{align*}
	\chi(G,G)&=-\left\langle v(G),v(G) \right\rangle=-\left\langle v(E^{*})+v(I_Z),v(E^{*})+v(I_Z) \right\rangle\\
	&=-\left\langle v(E^{*}),v(E^{*}) \right\rangle-2\left\langle v(E^{*}),v(I_Z) \right\rangle-\left\langle v(I_Z),v(I_Z) \right\rangle\\
	&=-\left\langle v(E),v(E) \right\rangle+2\chi(E^{*},I_Z)+\chi(I_Z,I_Z)\\
	&=-\left\langle v(E),v(E) \right\rangle+2(\chi(E)-rk)+(2-2k)\\
	&=2\left(-\frac{v(E)^2}{2}+\chi(E)-(r+1)k+1 \right) 
\end{align*}
Using the inequality \eqref{ineq} shows that we have $\chi(G,G)\geqslant 4$. But this implies 
\begin{equation*}
	\mathrm{hom}_X(G,G)+\mathrm{ext}^2_X(G,G)\geqslant 4.
\end{equation*}
One more application of Serre duality thus gives $\mathrm{hom}_X(G,G)\geqslant 2$. But this impossible, since $G$ is simple, as it is $\mu_h$-stable. So we do have $\mathrm{H}^1(X,E\otimes I_Z)=0$.
\end{proof}

\begin{cor}\label{eulch}
	Let $E$ be a locally free $\mu_h$-stable sheaf with Mukai vector $v=(r,h,s)$ satisfying the inequality \eqref{ineq}, then
	\begin{align*}
		\chi(E)&=\mathrm{h}^0(X,E)=r+s\\
		\chi(E\otimes I_Z)&=\mathrm{h}^0(X,E\otimes I_Z)=r+s-rk.
	\end{align*}
\end{cor}
	
We are now ready to study the image of a $\mu_h$-stable locally free sheaf with Mukai vector $v=(r,h,s)$ under the $\mathbb{P}^{k-1}$-functor $\Phi$. A priori this object is just a complex in $\De(\Xk)$ but in our situation we have:
	\begin{lem}\label{def:KE}
			Let $E$ be a locally free $\mu_h$-stable sheaf with Mukai vector $v=(r,h,s)$ satisfying the inequality \eqref{ineq}, then $\Phi(E)$ is a locally free sheaf of rank $r+s-rk$ on $\Xk$.
	\end{lem}
	
	\begin{proof}
		
		Using cohomology and base change results, the vanishing in Lemma \ref{cohvan} implies 
		\begin{equation*}
		R^ip_{*}( q^{*}(E)\otimes \IZ)=0\,\,\,\text{for}\,\, i=1,2.
		\end{equation*}
		Consequently $\Phi(E)=p_{*}(q^{*}E\otimes \IZ)$ is a sheaf. Furthermore the map
		\begin{equation*}
			\Xk \rightarrow \mathbb{Z},\,\,\,[Z]\mapsto \mathrm{h}^0(X,E\otimes I_Z)
		\end{equation*}
		is constant. Therefore $p_{*}(q^{*}E\otimes \IZ)$, that is $\Phi(E)$, is indeed a locally free sheaf on $\Xk$. As the fiber at a point $[Z]$ is just $\mathrm{H}^0(X,E\otimes I_Z)$ the rank follows from Corollary \ref{eulch}.
	\end{proof}
	
Next we want to study the slope-stability of $\Phi(E)$. For this we recall that in our situation we have 
	\begin{equation*}
		\NS(\X)=\ZZ h_k\oplus\ZZ\delta,
	\end{equation*}
	where $h_k$ is the divisor on $\X$ induced by the divisor $h$ on $X$ and $2\delta$ is the exceptional divisor of the Hilbert-Chow morphism $\X\rightarrow X^{(k)}$.
	For any coherent sheaf $F$ on $X$ we denote the associated coherent tautological sheaf by
	\begin{equation*}\label{eqn:taut}
		F^{[k]}:=p_{*}\left(q^{*}F\otimes \OO_{\mathcal{Z}} \right). 
	\end{equation*}
	
	\begin{lem}
		\label{lem:Chernclass}
		Let $E$ be a locally free $\mu_h$-stable sheaf with Mukai vector $v=(r,h,s)$ satisfying the inequality \eqref{ineq}, then $c_1(\Phi(E))=-h_k+r\delta$.
	\end{lem}
	
	\begin{proof}
		By Lemma \ref{def:KE} we have $R^1p_{*}(q^{*}E\otimes \mathcal{I}_{\mathcal{Z}})=0$, thus there is an exact sequence:
		\begin{equation}\label{seq1}
			\begin{tikzcd}
				0 \arrow[r] & p_{*}(q^{*}E\otimes \IZ) \arrow[r] &  p_{*}q^{*}E \arrow[r] & p_{*}(q^{*}E\otimes \OO_{\mathcal{Z}}) \arrow[r] & 0
			\end{tikzcd}
		\end{equation}
		
		We have 
		\begin{equation*}
			p_{*}q^{*}E\cong \mathrm{H}^0(E)\otimes\OO_{\X}
		\end{equation*}
		and the sheaf $p_{*}(q^{*}E\otimes \OO_{\mathcal{Z}})$ is by definition the tautological bundle $E^{[k]}$. The exact sequence \eqref{seq1} can be rewritten as
		\begin{equation}
			\label{eqn:sheafK}
			\begin{tikzcd}
				0 \arrow[r] & \Phi(E) \arrow[r] & \mathrm{H}^0(E)\otimes \OO_{\X} \arrow[r] & E^{[k]} \arrow[r] & 0.
			\end{tikzcd}
		\end{equation}
		
		Using \cite[Lemma 1.5]{wandel} we get
		\begin{equation*}
			c_1(\Phi(E))=-c_1(E^{[k]})=-c_1(E)_k+r\delta=-h_k+r\delta.
		\end{equation*}
	\end{proof}

	We also recall the notations introduced by Stapleton in \cite[Section 1]{stapleton}. The ample divisor $h$ on $X$ induces the ample divisor
	\begin{equation*}
		h_{X^k} =\bigoplus\limits_{i=1}^k q_i^\ast h
	\end{equation*}
	on $X^k$, where $q_i$ is the $i$-th projection from $X^k$, as well as a semi-ample divisor $h_k$ on $\X$. 
	
	We denote by $X^k_\circ$, $S^kX_\circ$ and $\X_\circ$ the loci of the relevant spaces parametrizing distinct points. The natural map
	\begin{equation*}
		\overline{\sigma}_\circ: X^k_\circ \to \X_\circ
	\end{equation*}
	is an \'etale cover and $j: X^k_\circ \to X^k$ is an open embedding. Given a coherent sheaf $F$ on $\X$, we denote by $F_\circ$ the restriction of $F$ to $\X_\circ$, and define
	\begin{equation*}
		(F)_{X^k} = j_\ast (\overline{\sigma}_\circ^\ast(F_\circ))
	\end{equation*}
	which is a torsion free coherent sheaf on $X^k$ if $F$ is on $\Xk$. Coherence follows since the complement of $X^k_{\circ}$ has codimension two, see e.g. \cite[Th\'{e}or\`{e}me 1, Th\'{e}or\`{e}me 2]{serre}. It is torsion free since $j: X^k_{\circ}\hookrightarrow X^k$ is dominant, see \cite[Proposition 7.4.5]{groth}.
	
	\begin{lem}\label{globsec}
		Let $E$ be a locally free $\mu_h$-stable sheaf with Mukai vector $v=(r,h,s)$ satisfying the inequality \eqref{ineq}, then
		\begin{equation*}
			\mathrm{H}^0(X^k,\left( \Phi(E)\right)_{X^k} )=0
		\end{equation*}
	\end{lem}

\begin{proof}
	Note that $(-)_\circ$ and $\overline{\sigma}_\circ^\ast(-)$ are exact, and $j_\ast(-)$ is left exact. Applying these functors to \eqref{eqn:sheafK} we obtain an exact sequence of $\mathfrak{S}_k$-invariant reflexive sheaves on $X^k$:
			\begin{equation}
		\begin{tikzcd}
			0 \arrow[r] & \left( \Phi(E)\right) _{X^k} \arrow[r] & \left( \mathrm{H}^0(E)\otimes \OO_{\X}\right)_{X^k}  \arrow[r,"\varphi"] & \left( E^{[k]}\right)_{X^k}
		\end{tikzcd}
	\end{equation}
where $\varphi$ may not be surjective. Certainly we have 
$$ (\mathrm{H}^0(E)\otimes \OO_{\X})_{X^k} = \mathrm{H}^0(E) \otimes \OO_{X^k}, $$
and also by \cite[Lemma 1.1]{stapleton}
$$ (E^{[k]})_{X^k} = \bigoplus\limits_{i=1}^k q_i^\ast E.$$
The above sequence can be written as
			\begin{equation}	\label{eqn:equisheaf}
	\begin{tikzcd}
		0 \arrow[r] & \left( \Phi(E)\right) _{X^k} \arrow[r] &  \mathrm{H}^0(E)\otimes \OO_{X^k}  \arrow[r,"\varphi"] & \bigoplus\limits_{i=1}^k q_i^{*}E.
	\end{tikzcd}
\end{equation} 

More accurately, $\varphi$ is the evaluation map on $X^k_\circ$: for any $k$-tuple $(x_1,\ldots, x_k) \in X^k$ of closed points with $x_i \neq x_j$, the morphism of fibers can be written as
\begin{align*}
	\varphi_{(x_1,\ldots, x_k)}: \mathrm{H}^0(E) &\longrightarrow \bigoplus\limits_{i=1}^k E_{x_i} \\
	s &\longmapsto (s(x_1),\ldots, s(x_k))
\end{align*}
Since for a non-trivial section $s \in \mathrm{H}^0(E)$, one can always choose a $k$-tuple of distinct points $(x_1,\ldots x_k) \in X^k$ with $(s(x_1),\ldots, s(x_k)) \neq (0,\ldots, 0)$, we see that the induced map
$$ \mathrm{H}^0(\varphi): \mathrm{H}^0(E) \longrightarrow \mathrm{H}^0(\bigoplus\limits_{i=1}^k q_i^\ast E) $$
is injective. It follows by exact sequence \eqref{eqn:equisheaf} that $\left(\Phi(E) \right) _{X^k}$ has no global sections.
\end{proof}
	
	\begin{prop}
		\label{prop:stableK}
		Let $E$ be a locally free $\mu_h$-stable sheaf with Mukai vector $v=(r,h,s)$ satisfying the inequality \eqref{ineq}, then the locally free sheaf $\Phi(E)$ defined in Lemma \ref{def:KE} is slope stable with respect to $h_k$.
	\end{prop}
	
	\begin{proof}
		We follow the proof of \cite[Theorem 1.4]{stapleton} and start with the exact sequence
					\begin{equation*}
			\begin{tikzcd}
				0 \arrow[r] & \left( \Phi(E)\right) _{X^k} \arrow[r] &  \mathrm{H}^0(E)\otimes \OO_{X^k}  \arrow[r,"\varphi"] & \bigoplus\limits_{i=1}^k q_i^{*}E
			\end{tikzcd}
		\end{equation*} 
		and note that $\varphi$ is surjective on $X^k_\circ$, hence its cokernel $\cok(\varphi)$ is supported on the big diagonal of $X^k$, which is of codimension $2$. We get
		$$ c_1(\left(\Phi(E) \right)_{X^k}) = -\sum\limits_{i=1}^k q_i^\ast h. $$
		
		Now assume $G$ is a reflexive subsheaf of $\Phi(E)$. Then $(G)_{X^k}$ is an $\mathfrak{S}_k$-invariant reflexive subsheaf of $\left(\Phi(E) \right)_{X^k}$. By \cite[Lemma 1.2]{stapleton} we have
		\begin{equation*}
			\mu_{h_{X^k}}((G)_{X^k}) < \mu_{h_{X^k}}(\left(\Phi(E) \right)_{X^k}) \Leftrightarrow \mu_{h_k}(G) < \mu_{h_k}(\Phi(E)).
		\end{equation*}
	It is therefore enough to prove that $\left(\Phi(E) \right)_{X^k}$ has no $\mathfrak{S}_k$-invariant destabilizing subsheaf (with respect to $h_{X^k}$). Assume $F$ is an $\mathfrak{S}_k$-invariant subsheaf, then we find:
		\begin{equation*}
			c_1(F) = a(\sum\limits_{i=1}^k q_i^\ast h)\,\,\,\text{for some $a \in \ZZ$}.
		\end{equation*}
			
		If $a \leqslant -1$, then 
		\begin{equation*}
			c_1(F) h_{X^k}^{2k-1} \leqslant c_1(\left(\Phi(E) \right)_{X^k}) h_{X^k}^{2k-1} < 0
		\end{equation*}
		Since $1 \leqslant \rk(F) < \rk(\left(\Phi(E) \right)_{X^k})$, we see that $\mu_{h_{X^k}}(F) < \mu_{h_{X^k}}(\left(\Phi(E) \right)_{X^k})$. Hence $F$ is not destabilizing.
		
		If $a=0$, we pick a (not necessarily $\mathfrak{S}_k$-invariant) non-zero stable subsheaf $F' \subseteq F$ that has maximal slope with respect to $h_{X^k}$ (for example one could take a stable factor in the first Harder-Narasimhan factor of $F$). Without loss of generality, we may assume $F$ and $F'$ are both reflexive. Since $F'$ is also a subsheaf of $\mathrm{H}^0(E)\otimes \OO_{X^k}$, there is a projection from $\mathrm{H}^0(E)\otimes \OO_{X^k}$ to a certain direct summand of it, such that the composition of the embedding and projection $$F' \rightarrow \mathrm{H}^0(E)\otimes \OO_{X^k} \rightarrow \OO_{X^k}$$ is non-zero. Since $$\mu_{X^k}(F') \geqslant \mu_{X^k}(F) = 0 = \mu_{X^k}(\OO_{X^k}),$$ and $\OO_{X^k}$ is also stable with respect to $h_{X^k}$, the map $F' \rightarrow \OO_{X^k}$ must be injective, and its cokernel is supported in codimension at least $2$. Since both sheaves are reflexive, we must have $F' = \OO_{X^k}$. As a result $F$, and consequently $\left(\Phi(E)\right)_{X^k}$, have non-trivial global sections, a contradiction to Lemma \ref{globsec}.
		
		If $a \geqslant 1$, $F$ would be a subsheaf of the trivial bundle $\mathrm{H}^0(E) \otimes \OO_{X^k}$ of positive slope which is not possible since a trivial bundle is semistable of slope zero.
	\end{proof}
	
	\begin{thm}\label{prop:sameH1}
		Let $E$ be a locally free $\mu_h$-stable sheaf with Mukai vector $v=(r,h,s)$ satisfying the inequality \eqref{ineq}, then $\Phi(E)$ is a locally free $\mu_H$-stable sheaf for some ample class $H \in \NS(\X)$ near $h_k$.
	\end{thm}
	
	\begin{proof}
		Proposition \ref{prop:stableK} and \cite[Theorem 2.3.1]{dCM} guarantee that the assumptions in \cite[Proposition 4.8]{stapleton} are satisfied for $\Phi(E)$, hence $\Phi(E)$ is slope stable with respect to some ample class $H$ near $h_k$ by \cite[Proposition 4.8]{stapleton}.		
	\end{proof}

\begin{rem}
\begin{enumerate}[a)]
\item Here we understand slope stability with respect to the non-ample divisor $h_k$ as slope stability with respect to a movable curve class. This stability was studied in detail by Greb, Kebekus and Peternell in \cite[Section 2.2]{greb16}. They show that all elementary properties satisfied by sheaves that are stable with respect to an ample polarization also hold for this notion of stability.
\item We restrict to locally free sheaves in this paper for two reasons: the first is to pass freely between $\mathrm{H}^{*}(E\otimes I_Z)$ and $\Ext_X^{2-*}(I_Z,E^{*})^{\vee}$ in Lemma \ref{cohvan}. The second reason is that we use Stapleton's description of $(E^{[k]})_{X^k}$, see \cite[Lemma 1.1]{stapleton}. This result uses the fact that $(E^{[k]})_{X^k}$ and $E^{\boxplus k}$ are reflexive (since they are locally free). But this fails if we start with a torsion free but not locally free sheaf $E$.
\end{enumerate}
\end{rem}

\section{A morphism of moduli spaces}
In the last section we saw that given a locally free $\mu_h$-stable sheaf $E$ with Mukai vector $v$ satisfying the inequality \eqref{ineq}, then there is an ample class $H\in \NS(\Xk)$ such that $\Phi(E)$ is a locally free $\mu_H$-stable sheaf on $\Xk$.

In this section we want to see that in certain cases we get a morphism 
\begin{equation*}
	\M_{X,h}(v) \rightarrow \mathcal{M},\,\,\,[E]\mapsto \left[ \Phi(E)\right]
\end{equation*}
for some moduli space $\mathcal{M}$ of stable sheaves on $\Xk$.

Let $v=(r,h,s)$ be a Mukai vector satisfying the following conditions:
\begin{enumerate}
	\item $v$ satisfies the inequality \eqref{ineq},
	\item all sheaves classified by $\M_{X,h}(v)$ are locally free and
	\item the moduli space $\M_{X,h}(v)$ is fine.
\end{enumerate}

Let $v=(r,h,s)$ be a Mukai vector that satisfies the conditions $(1)$ and $(2)$. Then for every $[E]\in \M_{X,h}(v)$ we know that there is an ample class $H$ such that $\Phi(E)$ is $\mu_H$-stable. One may ask how $H$ depends on $[E]$. This question is answered in the following theorem.

\begin{thm}\label{sameH2}
	Let $v$ be a Mukai vector that satisfies conditions $(1)$ and $(2)$, then there is an ample class $H\in \NS(\Xk)$ such that $\Phi(E)$ is $\mu_H$-stable for all $[E]\in \M_{X,h}(v)$ simultaneously.
\end{thm}

\begin{proof}
The proof is same as for \cite[Theorem 2.8]{rz2}. We just have to replace the sheaf $E_x$ by $\Phi(E)$ and the surface $X$ by the moduli space $\M_{X,h}(v)$. Then we note that the value of $$c = \mu_\beta(\Phi(E))$$ is independent of the choice of $[E] \in \M_{X,h}(v)$ by Lemma \ref{lem:Chernclass}. 

The finiteness of the set
$$ S := \{ c_1(F) \mid F \subseteq \Phi(E) \text{ for some } [E] \in \M_{X,h}(v) \text{ such that } \mu_\beta(F) \geqslant c \}. $$
can be seen as follows: using Corollary \ref{eulch} and exact sequence \eqref{eqn:sheafK} we see that every $\Phi(E)$ is a subsheaf of $\OO_{\Xk}^{r+s}$. Hence $S$ is a subset of 
$$ S':=  \{ c_1(F) \mid F \subseteq \OO_{\Xk}^{r+s} \text{ such that } \mu_\beta(F) \geqslant c \}. $$
But $S'$ is finite due to \cite[Theorem 2.29]{greb16}, so $S$ is also finite. The rest of the proof works unaltered.
\end{proof}

From now on we fix a Mukai vector $v$ that satisfies conditions $(1)-(3)$ and an ample class $H\in \NS(\Xk)$ that satisfies Theorem \ref{sameH2}. We denote by $\M_{\Xk,H}(\Phi^C(v))$ the moduli space of $\mu_H$-stable sheaves on $\Xk$ with Mukai vector $\Phi^C(v)$. Here
\begin{equation*}
\Phi^C: \mathrm{H}^{*}(X,\mathbb{Q}) \rightarrow \mathrm{H}^{*}(\Xk,\mathbb{Q})
\end{equation*}
is the induced cohomological Fourier-Mukai transform, see \cite[5.28, 5.29]{huy3}.

In the following we want to give an explicit construction of the morphism
\begin{equation*}
	\M_{X,h}(v) \rightarrow \M_{\Xk,H}(\Phi^C(v)),\,\,\,[E]\mapsto \left[ \Phi(E)\right].
\end{equation*}

We start by constructing a classifying family for the $\Phi(E)$.

\begin{lem}
There is a family $\mathcal{F}$ of sheaves on $\Xk\times \M_{X,h}(v)$, flat over $\M_{X,h}(v)$, such that for every $[E]\in \M_{X,h}(v)$ the restriction to the fiber over $[E]$ is given by
\begin{equation*}
\mathcal{F}_{[E]}= \Phi(E).
\end{equation*}
\end{lem}

\begin{proof}
As $\M_{X,h}(v)$ is fine, there is a universal family $\mathcal{U}$ on $X\times \M_{X,h}(v)$, flat over $\M_{X,h}(v)$. Lemma \ref{def:KE} shows that the complex $\Phi(\mathcal{U}_{[E]})=\Phi(E)$ is in fact a sheaf on $\Xk$ for every $[E]\in \M_{X,h}(v)$. The existence of the family $\mathcal{F}$ now follows from \cite[Proposition 4.2]{bridge}. 
\end{proof}

The family $\mathcal{F}$ induces a classifying morphism
	\begin{equation*}\label{eqn:class}
		f : \M_{X,h}(v) \longrightarrow \M_{\Xk,H}(\Phi^C(v)), \quad [E]\longmapsto [\Phi(E)].
	\end{equation*}
	
	The morphism $f$ has the following property:
	
	\begin{thm}\label{thm:component1}
		The family $\mathcal{F}$ identifies $\M_{X,h}(v)$ with a smooth connected component of $\M_{\Xk,H}(\Phi^C(v))$.
	\end{thm}
	
	\begin{proof}
	Using the family $\mathcal{F}$ as the kernel of an integral functor, we get an induced morphism
		\begin{equation*}
		\Ext^1_{\M_{X,h}(v)}(\mathcal{O}_{[E]},\mathcal{O}_{[E]}) \rightarrow \Ext^1_{\Xk}(\mathcal{F}_{[E]},\mathcal{F}_{[E]})
		\end{equation*}
		which is the Kodaira-Spencer map of the family $\mathcal{F}$ at the point $[E]\in \M_{X,h}(v)$ by \cite[Proposition 4.4]{bridge}.
		Using the canonical identifications
		\begin{equation*}
		\Ext^1_{\M_{X,h}(v)}(\mathcal{O}_{[E]},\mathcal{O}_{[E]})\cong T_{[E]}\M_{X,h}(v)\cong \Ext_X^1(E,E),
		\end{equation*}
		one can rewrite the induced morphism as
		\begin{equation*}
		\Ext^1_X(E,E)\rightarrow \Ext^1_{\Xk}(\Phi(E),\Phi(E))
		\end{equation*}
		which is an isomorphism by Remark \ref{extiso}.
		
This implies that the classifying map $f : \M_{X,h}(v) \longrightarrow \M_{\Xk,H}(\Phi^C(v))$ is \'{e}tale and surjective onto a smooth connected component. But using again Remark \ref{extiso} shows that $f$ has to be of degree one since for $[E]\neq [F]\in \M_{X,h}(v)$ we have
\begin{equation*}
\Hom_{\Xk}(\Phi(E),\Phi(F))\cong \Hom_X(E,F)=0.
\end{equation*}
So $f$ must be an isomorphism to a smooth connected component.
	\end{proof}
	
	\begin{rem}
	This argument uses essentially the same arguments as the proof in \cite[Section 1]{lange} where a component in a moduli space of stables bundles on a moduli space of bundles on a curve is constructed.
	\end{rem}		

Theorem \ref{sameH2} and Theorem \ref{thm:component1} are not void, that is there are Mukai vectors on certain K3 surfaces that satisfy conditions $(1)-(3)$. Here we give two examples:

\begin{exam}\phantomsection\label{exmp}
\begin{enumerate}[a)]
\item Let $X$ be a K3 surface with $\NS(X)=\mathbb{Z}h$ such that $h^2=50$. For $k=2$ the Mukai vector $v=(3,h,8)$ satisfies the conditions $(1)-(3)$: we have $\chi=11$. On the other hand $v^2=h^2-48=2$ and $(r+1)k=8$ so that  $\chi=11 \geqslant 10=\frac{v^2}{2}+(r+1)k+1$. A further computation shows $v^2+2 = 4 < 6 =2r$ so that all sheaves in $\M_{X,h}(v)$ are locally free by \cite[Lemma 4.4.2]{wray}. As $\mathrm{gcd}(3,8)=1$ we see that $\M_{X,h}(v)$ is a fine moduli space by \cite[Corollary 4.6.7]{huy}. In this case we have $\dim(\M_{X,h}(v))=4$ which gives a 4-dimensional component in $\M_{X^{[2]},H}(\Phi^C(v))$.
\item Let $X$ be a K3 surface with $\NS(X)=\mathbb{Z}h$ such that $h^2=186$. Then a similar computation shows that for $k=3$ then the Mukai vector $v=(5,h,18)$ satisfies the conditions $(1)-(3)$. In this example $\dim(\M_{X,h}(v))=8$ which gives an 8-dimensional component in $\M_{X^{[3]},H}(\Phi^C(v))$.
\end{enumerate}
\end{exam}

\end{document}